\documentclass[12pt]{amsart}
\usepackage{amssymb}
\usepackage{amsmath, amscd}
\usepackage{amsthm}
\usepackage[table,xcdraw]{xcolor}
\usepackage{ulem}
\usepackage{tikz}
\usepackage{circuitikz}
\usepackage[colorlinks=true, linkcolor=blue,urlcolor=blue]{hyperref}

\newtheorem{theorem}{Theorem}[section]

\newtheorem{proposition}[theorem]{Proposition}
\newtheorem{lemma}[theorem]{Lemma}

\newtheorem{corollary}[theorem]{Corollary}
\theoremstyle{definition}

\newtheorem{example}[theorem]{Example}

\newcommand{\bigzero}{\mbox{\normalfont\Large\bfseries 0}}

\begin{document}

\author[Peter Danchev]{Peter Danchev}
\address{Institute of Mathematics and Informatics, Bulgarian Academy of Sciences, 1113 Sofia, Bulgaria}
\email{danchev@math.bas.bg; pvdanchev@yahoo.com}

\author[A. Javan]{Arash Javan}
\address{Department of Mathematics, Tarbiat Modares University, 14115-111 Tehran Jalal AleAhmad Nasr, Iran}
\email{a.darajavan@modares.ac.ir; a.darajavan@gmail.com}

\author[O. Hasanzadeh]{Omid Hasanzadeh}
\address{Department of Mathematics, Tarbiat Modares University, 14115-111 Tehran Jalal AleAhmad Nasr, Iran}
\email{hasanzadeomiid@gmail.com}

\author[A. Moussavi]{Ahmad Moussavi}
\address{Department of Mathematics, Tarbiat Modares University, 14115-111 Tehran Jalal AleAhmad Nasr, Iran}
\email{moussavi.a@modares.ac.ir; moussavi.a@gmail.com}

\title{{\small Rings with} \large$u-1$ {\small quasinilpotent for each unit} \large$u$}

\maketitle

%\date{.}

%\begin{document}

%\maketitle

\begin{abstract}
We define and explore in-depth the notion of {\it UQ rings} by showing their important properties and by comparing their behavior with that of the well-known classes of UU rings and JU rings, respectively. Specifically, among the other established results, we prove that UQ rings are always Dedekind finite (often named directly finite) as well as that, for semipotent rings $R$, the following equivalence hold: $R/J(R)$ is UQ $\iff$ $R$ is UQ having the property that the set $QN(R)$ of quasinilpotent elements of $R$ coincides with the Jacobson radical $J(R)$ of $R$.
\end{abstract}

\section{Introduction and Motivation}

Throughout the present paper, all rings considered are associative and unital. For such a ring $R$, the notations of Jacobson radical, the group of units, the set of nilpotent elements, the set of quasinilpotent elements and the set of all idempotent elements of $R$ are denoted by $J(R)$, $U(R)$, $N(R)$, $QN(R)$ and $Id(R)$, respectively. For any subring $S$ of $R$, we shall almost always assume that $1_R = 1_S$, which means that the subring is unital. If, however, it is {\it not} the case, we hereafter will explicitly indicate this in the text. As usual, the used terminology is standard being in agreement with the classical book \cite{L}.

An element $a$ of a ring $R$ is said to be {\it quasinilpotent} if $1-ax$ is invertible for every $x \in R$ with $xa = ax$. It is thereby clear that both $N(R)$ and $J(R)$ are contained in $QN(R)$. It is also worth noting that quasinilpotents play an important role in the investigation of the complicated structure of a Banach algebra $\mathcal{A}$. By means of quasinilpotents, some interesting concepts were introduced like these: (strongly) J-clean rings \cite{chen2010}, nil-clean rings \cite{diesl2013}, generalized Drazin inverses \cite{koliha1996}, quasipolar rings \cite{chen2012}, etc.

A ring $R$ is called a {\it UU ring} if $U(R) = 1 + N(R)$. This notion was introduced by C\v{a}lug\v{a}reanu \cite{CUU}, and furthermore studied in more details by Danchev and Lam in \cite{DL}. They showed that $R$ is strongly nil-clean if, and only i,f $R$ is an exchange UU ring. In \cite{karimi}, Zhou et al. demonstrated that $R$ is a UU ring if, and only if, every invertible element (i.e., a unit) in $R$ is uniquely nil-clean.

It is clear that $1 + J(R) \subseteq U(R)$. Motivated by this inclusion and the above idea of UU rings, in \cite{Drings} were introduced the concept of a {\it JU ring}, that is, a ring $R$ for which $U(R) = 1 + J(R)$. Further, this notion was also studied in \cite{kosan1}, defined as a {\it UJ ring}. In fact, it was shown there that $R$ is a UJ ring if, and only if, every clean element of $R$ is J-clean. The structural characterization of UJ rings and their potential applications in various aspects of non-commutative ring theory have been examined in several recent studies \cite{Drings,kosan1,kosannot,leroy}.

Inspired by what we have stated so far, we call a ring $R$ a {\it UQ ring} if $U(R) = 1 + QN(R)$. It is evident that UU and UJ rings are both UQ rings, but the converse is manifestly {\it not} generally true. In the current article, we shall provide below a concrete counterexample to substantiate this statement.

Next, we say that an element $a \in R$ is {\it (strongly) quasi nil-clean} if there exists $e \in Id(R)$ such that $a - e \in QN(R)$ (with the additional commuting condition $ea = ae$ in the strongly case). Let us recollect that an element $a \in R$ is said to be {\it generalized strongly Drazin invertible} (or just {\it gs-Drazin invertible} for short) (see, e.g., \cite{Mosic}) if there is an element $x \in R$ satisfying the requirements
$$xax = x, ax = xa, \textrm{and} a-ax \in QN(R).$$
Similarly, a ring $R$ is called {\it gs-Drazin invertible} if each its element is gs-Drazin invertible. In \cite{Gurgun}, G\"urg\"un showed that an element $a \in R$ is strongly quasi nil-clean if, and only if, $a$ is gs-Drazin invertible. Later on, in \cite{CM} it was illustrated that a Banach algebra $\mathcal{A}$ is gs-Drazin invertible if, and only if, for all $a \in A$, $a-a^2 \in QN(A)$.

Our motivating tool is to give a satisfactory description of UQ rings by comparing their crucial properties with these of UU and JU rings, respectively, as well as to find some new exotic properties of UQ rings that are not too characteristically frequently seen in the existing literature. 

Our organization to achieve the desired establishments is subsumed in the following manner: In the next second section, we list our examples and individual properties of UQ rings. In the subsequent third section, among the other obtained results, we establish two important theorems, namely Theorems~\ref{thm 2 primal} and \ref{semipotent}, as well as the interesting Lemma~\ref{dedkind finite}. In the next fourth section, we study when groups rings are UQ and thus we prove the enclosed Theorem~\ref{UQgroupring}. Moreover, we somewhat reverse it by proving Proposition~\ref{locallyfinite} and, in this direction, we reconfirm once the recognizability of Theorem~\ref{UJgroupring} which was originally obtained in \cite{kosangroup}. In the next final fifth section, we show the abundance of some more unexpected classes of UQ rings which depend on the complicated structure of certain matrices.

\section{Examples and Basic Properties of UQ Rings}

We begin our work with the following three constructions.

\begin{example}\label{first ex}
For any ring R, we have:

(1) $J(R), N(R) \subseteq QN(R)$, but the reverse relation does not necessarily hold. Let $S = M_2(\mathbb{F}_2)$ and $K = \mathbb{F}_2[[x]]$. Take
$A=\begin{pmatrix}
    1 & 1 \\
    1 & 1
\end{pmatrix} \in S$, then $A \notin  J(S)$ but $A \in QN(S)$ (see \cite[Example 4.3]{Wang}). On the other hand, $x \in QN(K)$ but $x \notin N(K)$.

(2) We always have $QN(R) \cap U(R) = \emptyset$ and $QN(R) \cap Id(R) = \emptyset$ (see \cite[Lemma 3.1]{Cui}).

(3) Every UJ ring and every UU ring are UQ rings, but the converse is not necessarily true. Let $S = \mathbb{F}_2\left\langle x,y \mid x^2=0 \right\rangle$, $K = \mathbb{F}_2[[x]]$ and $L=S \times K$. Then $S$ is a UQ ring but not a UJ ring (see \cite[Example 2.2]{kosan1} and \cite[Example 2.5]{DL}). Also, by Corollary \ref{cor five}, $K$ is a UQ ring but not a UU ring, since $1+x \in U(K)$ but $x \notin N(K)$. And $L$ is a UQ ring but neither a UJ nor a UU ring.
\end{example}

We call a subring $S$ of $R$ a {\it good subring} if $U(R) \cap S = U(S)$. Clearly, $R$ is a good subring of $R[x]$ and $R[[x]]$. However, $R[x]$ is obviously {\it not} necessarily a good subring of $R[[x]]$, because $1+x \in U(\mathbb{Z}[[x]]) \cap \mathbb{Z}[x]$, but $1+x \notin U(\mathbb{Z}[x])$.

\medskip

We continue our considerations with a series of technicalities needed for our successful presentation in the sequel.

\begin{lemma} \label{good subring}
Let $S$ be a good subring of $R$. Then $ QN(R) \cap S \subseteq QN(S)$.
\end{lemma}

\begin{lemma} \label{subring}
Let $S$ be a good subring of $R$, so if $R$ is a UQ ring. Then $S$ is also a UQ ring.
\end{lemma}

\begin{proof}
Let $u \in U(S)\subseteq U(R)=1 + QN(R)$, thus $u - 1 \in QN(R) \cap S \subseteq QN(S)$.
\end{proof}

\begin{lemma} \label{prod}
Let $(R_i)_{i\in I}$ be a family of rings for some index set $I$. Then $QN ( \prod_{i \in I}  R_i )$  $= \prod_{i \in I} QN( R_i )$.
\end{lemma}

\begin{proof}
Let $(a_i) \in QN(\prod_{i \in I} R_i)$ and assume that a$_ib_i = b_ia_i$. Therefore, $(a_i)(b_i) = (b_i)(a_i)$, implying that $1 - (a_i)(b_i) \in U(\prod_{i \in I} R_i)$. This implies that $1 - a_ib_i \in U(R_i)$. Hence, for every $i \in I$, we have $a_i \in QN(R_i)$.

Conversely, let $(a_i) \in \prod_{i \in I} QN(R_i)$ and assume that $(a_i)(b_i) = (b_i)(a_i)$. Therefore, $a_ib_i = b_ia_i$, which implies that $1 - a_ib_i \in U(R_i)$. This in turn implies that $1 - (a_i)(b_i) \in U(\prod_{i \in I} R_i)$. Hence, $(a_i) \in QN(\prod_{i \in I} R_i)$.
\end{proof}

\begin{lemma}\label{product}
A direct product $\prod_{i\in I} R_i$ of rings is UQ if and only if each $R_i$ is UQ.
\end{lemma}

\begin{proof}
We know that $(a_i) \in U(\prod R_i)$ exactly when $a_i \in U(R_i)$, and with the aid of Lemma \ref{prod} we derive $QN(\prod R_i)= \prod QN(R_i)$, as required.
\end{proof}

For a subring $S$ of a ring $R$, the set $$\mathcal{R}[R, S] := \{(r_1, \ldots , r_n, s, s, \ldots) : r_i \in R, s \in S, n \ge  1\},$$ with addition and multiplication defined componentwise, is a ring.

As an immediate consequence, we yield:

\begin{corollary}
The ring $\mathcal{R}[R, S]$ is UQ ring if, and only if, $R$ and $S$ are UQ rings.
\end{corollary}

\begin{lemma}\label{corner}
Let $R$ be a UQ ring and $e \in Id(R)$. Then, the corner subring $eRe$ is a UQ ring.
\end{lemma}

\begin{proof}
Let $u \in U(eRe)$ with inverse $v$. Therefore, $u + (1 - e) \in U(R)$ with inverse $v + (1 - e)$, so $u + (1 - e) \in 1 + QN(R)$ and $u - e \in QN(R) \cap eRe = QN(eRe)$ by \cite[Lemma 3.5]{chen2012}. Hence, $u \in e + QN(eRe) = 1_{eRe} + QN(eRe)$, and thus $eRe$ is a UQ ring, as claimed.
\end{proof}

Let $R$ be a ring and $M$ a bimodule over $R$. The trivial extension of $R$ and $M$ is defined as
\[ T(R, M) = \{(r, m) : r \in R \text{ and } m \in M\}, \]
with addition defined componentwise and multiplication defined by the equality
\[ (r, m)(s, n) = (rs, rn + ms). \]
Observe that the trivial extension $T(R, M)$ is isomorphic to the subring
\[ \left\{ \begin{pmatrix} r & m \\ 0 & r \end{pmatrix} : r \in R \text{ and } m \in M \right\} \]
consisting of the formal $2 \times 2$ matrix ring $\begin{pmatrix} R & M \\ 0 & R \end{pmatrix}$, and also $T(R, R) \cong R[x]/\left\langle x^2 \right\rangle$.

\medskip

A Morita context is defined as a 4-tuple $\begin{pmatrix} A & M \\ N & B \end{pmatrix}$, where $A$ and $B$ are rings, $_AM_B$ and $_BN_A$ are bimodules, and there exist context products $M\times N \to A$ and $N\times M \to B$ written multiplicatively as $(w, z) = wz$ and $(z, w) = zw$. The Morita context $\begin{pmatrix} A & M \\ N & B \end{pmatrix}$ forms an associative ring with the usual matrix operations.

A Morita context is referred to as trivial if the context products are trivial, meaning that $MN = 0$ and $NM = 0$ (see \cite[p. 1993]{Mari}). In this case, we have the isomorphism
$$\begin{pmatrix} A & M \\ N & B \end{pmatrix} \cong T(A \times B, M\oplus N),$$
where $\begin{pmatrix} A & M \\ N & B \end{pmatrix}$ represents a trivial Morita context as stated in \cite{Kosan2015tri}.

The following statement represents some critical set-theoretical inclusions.

\begin{lemma}\label{basic property}
Let $R$, $S$ be rings, $N$ be an $(R, S)$-bimodule, and $M$ a bimodule over $R$. Then we have:

(1) $\{(r, m)\in T(R, M)  : r \in QN(R) \text{ and } m \in M\} \subseteq QN(T(R, M))$.

(2) $\{\begin{pmatrix} r & m \\ 0 & s \end{pmatrix}: r \in QN(R), s \in QN(S) \text{ and } m \in N\} \subseteq QN(\begin{pmatrix} R & N \\ 0 & S \end{pmatrix})$.

(3) $\{(a_{ij}) \in T_n(R) : a_{ii} \in QN(R) \text{ for all } 1 \le i \le n \} \subseteq QN(T_n(R))$.

(4) $\{a_0+ \cdots a_{n-1}x^{n-1} \in R[x]/\left\langle x^n \right\rangle : a_0 \in QN(R)\} \subseteq QN(R[x]/\left\langle x^n \right\rangle)$.

(5) $\{a_0+ a_1x+ a_2x^2 +\cdots \in R[[x]] : a_0 \in QN(R)\} \subseteq QN(R[[x]]$.

In the above relations, equality holds when $R$ and $S$ are UQ rings.
\end{lemma}

\begin{proof}
(1) Assuming $(r,m) \in T(QN(R),M)$ and $(s,n)(r,m) = (r,m)(s,n)$, since $r \in QN(R)$ and $rs = sr$, so we must have $1-rs \in U(R)$, it follows that $(1,0) - (r,m)(s,n) = (1-rs, \ast) \in T(U(R),M) = U(T(R,M))$.
Now, if $R$ is a UQ ring, then for every $(r,m) \in QN(T(R,M))$, we have $1+r \in U(R) = 1+QN(R)$, so $r \in QN(R)$.

(2) We know that $\begin{pmatrix} R & N \\ 0 & S \end{pmatrix} \cong T(R \times S, N)$, so point (1) along with Lemma \ref{prod} allow us to see that nothing remains to be proven.

(3) It is sufficient to substitute $S = T_{n-1}(R)$ and $N = R^{n-1}$ in (2).

(4) We know that
$$R[x]/\left\langle x^n \right\rangle \cong \left \{ \begin{pmatrix}
    a_1 & a_2 & a_3 & \cdots & a_{n} \\
    0 & a_1 & a_2& \cdots & a_{n-1} \\
    0& 0& a_1& \cdots & a_{n-2} \\
    \vdots& \vdots& \vdots& \ddots&\vdots \\
    0& 0& 0& \cdots & a_1
\end{pmatrix} \in T_n(R) : a_i \in R\right \}.$$
So, it is clear that $R[x]/\left\langle x^n \right\rangle$ is a good subring of $T_{n}(R)$. Therefore, according to point (3) and Lemma \ref{good subring}, the proof is self-evident. 

Reciprocally, if $R$ is a UQ ring and $$f=a_0 +a_1x+ \cdots a_n \in QN(R[x]/\left\langle x^n \right\rangle),$$ then $1 - a_0 \in U(R) = 1 + QN(R)$, so $a_0 \in QN(R)$.

(5) Assuming $$f=\sum_{i=0}a_ix^i \in QN(R) + R[[x]]x$$ and $$g=\sum_{i=0}b_ix^i \in R[[x]]$$ such that $fg = gf$, we find that $a_0b_0 = b_0a_0$ whence $1 - a_0b_0 \in U(R)$. Consequently, $$1 - fg \in U(R) + R[[x]]x = U(R[[x]]).$$ Thus, $f \in QN(R[[x]])$. Now, if $R$ is a UQ ring and $f \in QN(R[[x]])$, we have $$(1 + a_0) + \sum_{i=1}a_ix^i = 1 + f \in U(R[[x]]) =$$ 
$$U(R) + R[[x]]x = 1 + QN(R) + R[[x]]x.$$ This, in turn, forces that $a_0 \in QN(R)$, as required.
\end{proof}

As a direct consequence, we extract:

\begin{corollary} \label{cor five}
Let $R$, $S$ be rings, $N$ be an $(R, S)$-bimodule, and $M$ a bimodule over $R$. Then we have:

(1) The trivial extension $T(R, M)$ is a UQ ring if and only if $R$ is a UQ ring

(2) The formal triangular matrix ring $\begin{pmatrix} R & N \\ 0 & S \end{pmatrix}$ is a UQ ring if and only if $R$ and $S$ are a UQ ring.

(3) For $n \ge 1$, $T_n(R)$ is a UQ ring if and only if $R$ is a UQ ring.

(4) For $n \ge 1$, $R[x]/\left\langle x^n \right\rangle$ is a UQ ring if and only if $R$ is a UQ ring.

(5) $R[[x]]$ is a UQ ring if, and only if, $R$ is a UQ ring.
\end{corollary}

\begin{proof}
We only need to establish (1), as the other cases can be proved similarly. Suppose $T(R,M)$ is a UQ ring, and we can consider $R$ as a good subring of $T(R,M)$ (because, $R \cong T(R,0)$). Therefore, by Lemma \ref{subring}, $R$ is a UQ ring. Now, conversely, if $R$ is a UQ ring and $(r,m) \in U(T(R,M))$, then $r \in U(R)$. Since $R$ is a UQ ring, we have $1-r \in QN(R)$. Thus, by Lemma \ref{basic property}, $(1,0)-(r,m) = (1-r,m) \in T(QN(R),M) \subseteq QN(T(R,M))$.
\end{proof}

\section{Principal Results}

We now come to our chief assertion which gives necessary and sufficient conditions in the following two situations. Standardly, the symbol $N_{\ast}(R)$ stands for the lower nil-radical of a ring $R$. Recall that a ring $R$ is termed {\it $2$-primal}, provided that $N(R)=N_{\ast}(R)$.

\begin{theorem} \label{thm 2 primal}
Let $R$ be a ring. Then, the following two items are true:

(1) $R$ is a $2$-primal ring if, and only if, $QN(R[x]) = N_{\ast}(R)[x]$.

(2) $R$ is a reduced ring if, and only if, $QN(R[x]) = \{0\}$.

\end{theorem}

\begin{proof}
(1) Let us assume $R$ is a $2$-primal ring. If $f \in N_{\ast}(R)[x]$, then according to the idealness of $N(R)$, for every $g \in R[x]$ such that $fg = gf$, we have $1 - fg(0) \in U(R)$ and $fg(n) \in N(R)$ for every $n$. Therefore, by \cite[Theorem 2.5]{Chenpr}, we conclude that $1 - fg \in U(R[x])$. Thus, $f \in QN(R[x])$.

Now, if $f \in QN(R[x])$, then $1 - xf \in U(R[x])$ since $xf = fx$. In virtue of \cite[Theorem 2.5]{Chenpr}, one derives that $a_i \in N(R)$ for every $0 \le i \le n$. Therefore, $f \in N(R)[x]=N_{\ast}(R)[x]$.

Conversely, let us assume that $QN(R[x]) = N_{\ast}(R)[x]$. If $a \in N(R)$ and $a^n = 0$, then $a^n \in QN(R[x])$. By virtue of \cite[Proposition 2.7(1)]{Cui}, one infers that $a \in QN(R[x]) = N_*(R)[x]$. Thus, $a \in N_*(R)$.

(2) Based on what we have proved in the first part, nothing remains to be showed.
\end{proof}

As two consequent corollaries, we deduce:

\begin{corollary}\label{cor 2 primal}
Let $R$ be a 2-primal ring. Then, we have:

\[ J(R[x]) = QN(R[x]) = \text{N}(R)[x] = \text{N}_*(R)[x] = \text{N}_*(R[x]) \]
\end{corollary}

\begin{proof}
Similarly to the proof of Theorem \ref{thm 2 primal}, it can be shown that $J(R[x]) = N(R)[x]$.
\end{proof}

\begin{corollary}
Let $R$ be a $2$-primal ring. Then the following conditions are equivalent:

(1) $R$ is a UU-ring.

(2) $R[x]$ is a UQ-ring.

(3) $R[x]$ is a UJ-ring.

(4) $R[x]$ is a UU-ring.
\end{corollary}

\begin{proof}
(1) $\Leftrightarrow$ (4). This is immediate directing to \cite[Theorem 2.11]{karimi}.

(2) $\Leftrightarrow$ (3). It follows at once from Corollary \ref{cor 2 primal}.

(1) $\Rightarrow$ (2). Let us assume $f=\sum_{i=0}^{n}a_ix^i \in U(R[x])$. Since $R$ is a $2$-primal ring, exploiting \cite[Theorem 2.5]{Chenpr}, we may write $a_0 \in U(R)$ and $a_i \in N(R)$ for all $1 \le i \le n$. However, since $R$ is a UU ring, we find that $$1 - f = (1 - a_0) + \sum_{i=1}^{n}a_ix^i \in N(R)[x] = QN(R[x]).$$

(2) $\Rightarrow$ (1). Let us assume $u \in U(R)$. Then, we can write that $$u \in U(R[x]) = 1 + QN(R[x]) = 1 + N(R)[x].$$ Therefore, $u - 1 \in \text{Nil}(R)$.
\end{proof}

Our next series of technical claims is as follows:

\begin{lemma} \label{close prod}
Let $R$ be a ring and let $b \in Z(R)$. Then, we have:

(1) If $a \in QN(R)$, then $ab \in QN(R)$.

(2) If $a$ and $b \in QN(R)$, then $a + b \in QN(R)$.
\end{lemma}

\begin{proof}
(1) Assume $x \in R$ such that $x(ab) = (ab)x$. So, $(bx)a = a(bx)$. Since $a \in QN(R)$, we have $1 - a(bx) \in U(R)$, which implies that $ab \in QN(R)$.

(2) Assume $x \in R$ such that $x(a+b) = (a+b)x$. Since $b \in Z(R)$, we have $ax = xa$ and $1 + xb \in U(R)$. Therefore, we can conclude that $((1 + xb)^{-1}x)a=a((1 + xb)^{-1}x)$. This implies that $1 + x(a+b) = (1 + xb)(1 + (1 + xb)^{-1}xa) \in U(R)$. Hence, $a + b \in QN(R)$.
\end{proof}

In the above lemma, the condition $b \in Z(R)$ is surely {\it not} redundant and cannot be removed. To substantiate this, let us consider for instance the setting: $R = M_2(\mathbb{Z})$, $a = e_{12}$, and $b = e_{21}$. According to Example \ref{first ex}(1), we have $a,b \in QN(R)$. However, based on Example \ref{first ex}(2), we see that $ab$ and $a+b$ are not in $QN(R)$. Therefore, it is {\it not} necessarily true that $QN(R)$ is a subring of $R$.

According to \cite[Lemma 1.1]{kosan1}, a ring $R$ is UJ if and only if $U(R) + U(R) = J(R)$. Similarly, we can state the following lemma:

\begin{lemma} \label{equ UQ}
$R$ is a UQ ring if, and only if, $U(R) + (U(R)\cap Z(R)) = QN(R)$
\end{lemma}

\begin{proof}
Evidently, $QN(R)$ is always a subset of $U(R) + (U(R)\cap Z(R))$. Now, if $R$ is a UQ ring and $u \in U(R)$ and $v \in U(R) \cap Z(R)$, then $1+u$, $1-v \in QN(R)$. Therefore, by Lemma \ref{close prod}, we have $u+v=(1+u)-(1-v) \in QN(R)$. The reverse is also evident since for every $u \in U(R)$, we have $u-1 \in QN(R)$.
\end{proof}

\begin{lemma}\label{matrix}
For any ring $S \neq 0$ and any integer $n \ge 2$, the matrix ring $M_n(S)$ of size $n$ is {\it not} a UQ ring.
\end{lemma}

\begin{proof}
Since $M_2(S)$ is isomorphic to a corner ring of $M_n(S)$ (for $n \ge 2$), it suffices to show that $M_2(S)$ is not a UQ ring. Consider the matrix
    $U=\begin{pmatrix}
        0 & 1\\1 & 1
    \end{pmatrix},$ where is a unit. Since
    $I-U=\begin{pmatrix}
        1 & -1\\-1 & 0
    \end{pmatrix}$
is a unit too, it cannot be quasinilpotent, so $M_2(S)$ is not a UQ ring, as asserteds.
\end{proof}

A set $\{e_{ij} : 1 \le i, j \le n\}$ of nonzero elements of $R$ is said to be a system of $n^2$ matrix units if $e_{ij}e_{st} = \delta_{js}e_{it}$, where $\delta_{jj} = 1$ and $\delta_{js} = 0$ for $j \neq s$. In this case, $e := \sum_{i=1}^{n} e_{ii}$ is an idempotent of $R$ and $eRe \cong M_n(S)$, where $$S = \{r \in eRe : re_{ij} = e_{ij}r,~~\textrm{for all}~~ i, j = 1, 2, . . . , n\}.$$

\begin{lemma} \label{dedkind finite}
Every UQ ring is Dedekind finite.
\end{lemma}

\begin{proof}
If $R$ is not a Dedekind finite ring, then there exist elements $a, b \in R$ such that $ab = 1$ but $ba \neq 1$. Assuming $e_{ij} = a^i(1-ba)b^j$ and $e =\sum_{i=1}^{n}e_{ii}$, there exists a nonzero ring $S$ such that $eRe \cong M_n(S)$. However, owing to Proposition \ref{corner}, $eRe$ is a UQ ring, so $M_n(S)$ must also be a UQ ring, which contradicts proposition \ref{matrix}.
\end{proof}

We, thus, infer the following consequence.

\begin{corollary}
Let $R$ be a UQ ring and let $a, b \in R.$ Then, $1-ab \in QN(R)$ if, and only if, $1-ba \in QN(R)$.
\end{corollary}

\begin{proof}
Assuming that $1-ab \in QN(R)$, we have $ab \in U(R)$. Therefore, invoking Lemma \ref{dedkind finite}, $a \in U(R)$. Since $1-ab=(1-ab)aa^{-1} \in QN(R)$, it follows from [2, Proposition 3] that $a^{-1}(1-ab)a \in QN(R)$. Hence, $1-ba=a^{-1}(1-ab)a \in QN(R)$, as needed.
\end{proof}

\begin{lemma} \label{sum two unit}
Let $R$ be a UQ ring and let $\overline{R} = R/J(R)$. The following hold:

(1) For any $u_1, u_2 \in U(R)$, $u_1 + u_2 \neq 1$.

(2) For any $\bar{u}_1, \bar{u}_2 \in U(\overline{R})$, $\bar{u}_1 + \bar{u}_2 \neq \bar{1}$.

\end{lemma}

\begin{proof}
(1) According to Lemma \ref{equ UQ} and Example \ref{first ex}, the proof is straightforward.

(2) Let us assume $\bar{u}_1 + \bar{u}_2 \neq \bar{1}$. Then, $u_1, u_2 \in U(R)$ and $u_1 + u_2 - 1 \in J(R)$. Therefore, $u_1 - 1 \in U(R) + J(R)$. However, since $U(R) + J(R) \subseteq U(R)$, there exists $u'_2 \in U(R)$ such that $u_1 + u'_2 = 1$, which contradicts (1).
\end{proof}

\begin{lemma}\label{2 in J(R)}
Let $R$ be a UQ ring. Then, the following conditions hold:

(1) $2 \in QN(R)$ and, in particular, $2 \in J(R)$.

(2) $x \in QN(R)$ precisely when $x^2 \in QN(R)$.

(3) $QN(R)$ is closed under addition uniquely when $QN(R)$ is a subring of $R$.

\end{lemma}

\begin{proof}
(1) It is elementary consulting with Lemma \ref{equ UQ}.

(2) If $x \in QN(R)$, then $1\pm x \in U(R)$. Therefore,
$$1-x^2 = (1-x)(1+x) \in U(R) = 1+QN(R),$$
so $x^2 \in QN(R)$. Now, if $x^2 \in QN(R)$, then
$$(1-x)^2 = (1+x^2) - 2x \in U(R) + J(R) \subseteq U(R),$$
which implies $1-x \in U(R) = 1+QN(R)$, thus $x \in QN(R)$.

(3) Assuming that QN(R)is closed under addition, it is sufficient to show that QN(R) is closed under multiplication. Let's assume $a, b \in QN(R)$. Then, we have
$(1+x)(1+y) \in U(R) = 1+QN(R)$, which gives $x+y+xy \in QN(R)$. Moreover, from the assumption that $x+y\in QN(R)$, it follows that $xy \in QN(R)$. Therefore, QN(R) is closed under multiplication.
\end{proof}

Concretely, we obtain:

\begin{corollary}
The ring $\mathbb{Z}_n$ is UQ if, and only if, $n$ is a power of 2.
\end{corollary}

\begin{proof}
If $\mathbb{Z}_n$ is UQ ring, then $2 \in J(R)$, so it must be the case that $n=2^k$ for some $k \in \mathbb{Z}$.

Conversely, if $n=2^k$, then it is clear that $QN(R)=\{a \in \mathbb{Z}_{2^k} : a \textit{ is even}\}$ and $U(R)=\{a \in \mathbb{Z}_{2^k} : a \text{ is odd}\}$.

\end{proof}

\begin{lemma}
Let $R$ be a UQ ring. Then, the following conditions hold:

(1) If $R$ is a division ring, then $R \cong \mathbb{F}_2$.

(2) If $R$ is a local ring, then $R/J(R) \cong \mathbb{F}_2$. (In particular, a local ring $R$ is UQ if, and only if, $R$ is uniquely clean ring).

(3) If $R$ is a semisimple ring, then $R \cong \mathbb{F}_2 \times \cdots \times \mathbb{F}_2$.

\end{lemma}

\begin{proof}
(1) If $R$ is a division ring, then we know that $QN(R)\setminus \{0\} \subseteq U(R)$, and since $QN(R) \cap U(R)=0$ it must be $QN(R)=0$. Therefore, $U(R)={1}$.

(2) If $R$ is a local ring, then $\overline{R}=R/J(R)$ is a division ring. For every $\bar{u} \in U(\overline{R})\setminus\{\bar{1}\}$, we have $\bar{1}-\bar{u} \in U(\overline{R})$. So, $\bar{u}+(\bar{1}-\bar{u})=\bar{1}$, which contradicts Lemma \ref{sum two unit}. Therefore, $\bar{u}=\bar{1}$.

(3) If $R$ is a semisimple ring, in accordance with the well-known Wedderburn–Artin theorem we write $R\cong \prod M_{n_i}(D_i)$. Consequently, based on Lemma \ref{product} and \ref{matrix}, we conclude $R \cong \mathbb{F}_2 \times \cdots \times \mathbb{F}_2$, as expected.
\end{proof}

Our next main assertion is the following criterion which gives some connections between UU, UJ and UQ semipotent rings, respectively.

\begin{theorem}\label{semipotent}
Let $R$ be a semipotent ring. The following are equivalent:

(1) The ring $R/J(R)$ is a UQ ring.

(2) The ring $R/J(R)$ is a Boolean ring.

(3) The ring $R$ is a UJ ring.

(4) The ring $R/J(R)$ is a UU ring.

(5) The ring $R$ is a UQ ring such that $QN(R)=J(R)$.
\end{theorem}

\begin{proof}
(1) $\Rightarrow$ (2).  Since $R$ is a semipotent ring, we have $\overline{R} = R/J(R)$, which is also a semipotent ring. We will prove that $\overline{R}$ is a reduced ring.

Assume $x^2 = 0$ but $0 \neq x \in \overline{R}$. Then, by \cite[Theorem 2.1]{Levi}, there exists $e \in \overline{R}$ such that $e\overline{R}e \cong M_2(S)$, where $S$ is a nonzero ring. However, since $\overline{R}$ is a UQ ring (by assumption), by Lemma \ref{corner}, $e\overline{R}e$ is a UQ ring. But on the other hand, by Lemma \ref{matrix}, $M_2(S)$ is not a UQ ring, which is a contradiction. Therefore, $\overline{R}$ is a reduced ring.

Now, assume there exists $x \in \overline{R}$ such that $x - x^2 \neq 0$ in $\overline{R}$. Since $\overline{R}$ is a semipotent ring, there exists $e = e^2 \in \overline{R}$ such that $e \in (x - x^2)\overline{R}$. So, $e = (x - x^2)y$ for some $y \in \overline{R}$. Since $e$ is central (indeed, as $[er(1-e)]^2 = 0 = [(1-e)re]^2$, so we have $er(1-e) = 0 = (1-e)re$), we can write $e = ex \cdot e(1-x) \cdot ey$, so that both $ex, e(1-x) \in U(e\overline{R}e)$. But, by Lemma \ref{corner}, we know that $e\overline{R}e$ is a UQ ring. However, $ex + e(1-x) = e$, which contradicts Lemma \ref{sum two unit}. Therefore, $\overline{R}$ is a Boolean ring.

(2) $\Rightarrow$ (3). Let us assume $u \in U(R)$. Then, $\Bar{u} \in U(\overline{R})$. Since $\overline{R}$ is a Boolean ring, we have $\bar{u} = \bar{1}$, which implies $u - 1 \in J(R)$.

(3) $\Rightarrow$ (4). Let us assume $\bar{u} \in U(\overline{R})$. Then, $u \in U(R)$. Since $R$ is a UJ ring, we have $u = 1 + j$, where $j \in J(R)$. Therefore, $\bar{u} = \bar{1}+ \bar{0}$.

(4) $\Rightarrow$ (1). It is apparent, because as already noted above $N(R) \subseteq QN(R)$.
\end{proof}

Surprisingly, we arrive at the following consequence.

\begin{corollary}
A (von Neumann) regular ring $R$ is UQ if, and only if, $R$ is UJ if, and only if, $R$ is UU if, and only if, $R$ is Boolean.
\end{corollary}

Our further assertions of this aspect are the following.

\begin{proposition} \label{pro clean element}
For any UQ ring $R$, the following statements are true:

(1) An element $a$ in R is clean if, and only if, it is quasi nil-clean.

(2) An element $a$ in R is strongly clean if, and only if, it is strongly quasi nil-clean.
\end{proposition}

\begin{proof}
We will only prove part (1), as the other case is similar.

Let us assume that $a \in R$ is clean, and $a = e + u$ is a clean decomposition. Since $R$ is a UQ ring, we have $u = 1 + q$, where $q \in QN(R)$. Therefore, $r = (1 - e) + (2e + q)$. On one hand, since $(1 + q) + 2e \in U(R) + J(R) \subseteq U(R) = 1 + QN(R)$, we conclude that $2e + q \in QN(R)$, as required.

Conversely, if $a \in R$ is $q$-nil-clean and $a = e + q$ is a $q$-nil-clean decomposition, we have $a = (1 - e) - (1 - (2e + q))$. Similar to the previous case, we can easily show that $1 - (2e + q) \in U(R)$.
\end{proof}

Incidentally, we receive:

\begin{corollary} \label{impor cor equ}
Suppose that $R$ is an arbitrary ring. Then, the following conditions are equivalent:

(1) $R$ is a UQ ring.

(2) Every clean element is a strongly quasi nil-clean.

(3) For every $u \in U(R)$, there exists $e = e^2 \in Z(R)$ and $q \in QN(R)$ such that $u = e + q$.
\end{corollary}

\begin{proof}
(1) $\Rightarrow$ (2). it follows at once from Proposition \ref{pro clean element}.

(2) $\Rightarrow$ (1). Given $ u \in U(R)$. Since $u$ is strongly clean, by (2) we have $u = e + q$, where $e \in Id(R)$, $q \in QN(R)$, and $eq = qe$. Thus, $ u^{-1}e = 1 + u^{-1}q \in U(R)$. This ensures that $e \in U(R) \cap Id(R) = \{1\}$.

(1) $\Rightarrow$ (3). It suffices to assume $e = 1$, which manifestly yields our claim.

(3) $\Rightarrow$ (1). Assume $u \in U(R)$. So, $u = e + q$ and, therefore, $u^{-1}e = 1 + u^{-1}q \in U(R)$. This shows that $e = 1$.
\end{proof}

We, thereby, conclude:

\begin{corollary}\label{cor1}
Each strongly quasi nil-clean ring is UQ.
\end{corollary}

\begin{proof}
Indeed, the result can be obtained directly from Corollary \ref{impor cor equ}(2).
\end{proof}

\begin{lemma}\label{cor2}
Any strongly quasi nil-clean ring is strongly clean.
\end{lemma}

\begin{proof}
Let $a \in R$ be strongly quasi nil-clean. Then, there exist $e^2 = e \in R$ and $q \in QN(R)$ such that $a = e + q$ and $eq = qe$. Thus, we have $a = (1 - e) + (2e - 1 + q)$. With Lemma \ref{2 in J(R)} at hand, we have $2 \in J(R)$, so $2e -1 + q \in J(R) + U(R) \subseteq U(R)$.
\end{proof}

We now immediately obtain that:

\begin{corollary}
A ring $R$ is strongly quasi nil-clean if, and only if, the following two points are fulfilled:

(1) $R$ is UQ.

(2) $R$ is strongly clean.
\end{corollary}

\begin{proof}
If $R$ is a strongly quasi nil-clean ring, the claim is concluded from Corollary \ref{cor1} and Lemma \ref{cor2}. Now, if $R$ is a UQ and strongly clean ring, Proposition \ref{pro clean element} applies to get that $R$ is a strongly quasi nil-clean ring, as promised.
\end{proof}

\begin{lemma}\label{exe}
Let $R$ be a potent UQ ring, and $\overline{R} = R/J(R)$. Then, we have:

(1) For any $\bar{e}=\bar{e}^2 \in \overline{R}$ and any  $\bar{u}_1,  \bar{u}_2 \in U(\bar{e}\overline{R}\bar{e})$, $\bar{u}_1 +  \bar{u}_2 \neq \bar{e}$.

(2) There does not exist $\bar{e}=\bar{e}^2 \in \overline{R}$ such that $\bar{e}\overline{R}\bar{e} \cong M_2(S)$ for some ring $S$.
\end{lemma}

\begin{proof}
(1) Given $\bar{e}$, $\bar{u}_1$,$\bar{u}_2$ as in (1), we can assume $e^2 = e \in R$ because idempotents lift modulo $J(R)$. Then $\bar{e}\overline{R}\bar{e} \cong eRe/J(eRe)$. Because $eRe$ is UQ by Lemma \ref{corner}, (1) follows directly by \ref{sum two unit}(2).

(2) In a $2 \times 2$ matrix ring, it is always true that
$$\begin{pmatrix} 1 & 0 \\ 0 & 1   \end{pmatrix}=
\begin{pmatrix} 1 & 1 \\ 1 & 0   \end{pmatrix}+
\begin{pmatrix} 0 & -1 \\ -1 & 1   \end{pmatrix}\in U(M_2(S))+U(M_2(S)).$$
Hence, there exist $\bar{u}_1, \bar{u}_2 \in U(\bar{e}\overline{R}\bar{e})$ such that $\bar{u}_1 + \bar{u}_2 = \bar{e}$. This is in contrary to (1).
\end{proof}

A pivotal consequence is the following.

\begin{corollary}\label{exe 2}
Let $R$ be a potent ring. Then, the following are equivalent:

(1) $R$ is a UQ ring.

(2) $R/J(R)$ is a UQ ring.

(3) $R/J(R)$ is a Boolean ring.

(4) $R$ is a UJ ring.

(5) $R/J(R)$ is a UJ ring.

(6) $R/J(R)$ is a UU ring.

\end{corollary}

\begin{proof}
It can be derived with the help of Theorem \ref{semipotent} that the implications (2) $\Leftrightarrow$ (3) $\Leftrightarrow$ (4) $\Leftrightarrow$ (6) are valid. Furthermore, (4) $\Leftrightarrow$ (5) utilizing \cite[Proposition 1.3 (5)]{kosan1}. It is also trivial that (4) $\Rightarrow$ (1). Therefore, it is sufficient to prove that (1) $\Rightarrow$ (3).

(1) $\Rightarrow$ (3). Since $R$ is a semipotent ring, we have $\overline{R} = R/J(R)$, which is also a semipotent ring. We will prove that $\overline{R}$ is a reduced ring.

Assume $x^2 = 0$ but $0 \neq x \in \overline{R}$. Then, by \cite[Theorem 2.1]{Levi}, there exists $\bar{e} \in \overline{R}$ such that $\bar{e}\overline{R}\bar{e} \cong M_2(S)$, where $S$ is a nonzero ring. This contradicts Lemma \ref{exe}(2). Therefore, $\overline{R}$ is a reduced ring.

Now, assume there exists $x \in \overline{R}$ such that $x - x^2 \neq 0$ in $\overline{R}$. Since $\overline{R}$ is a semipotent ring, there exists $e = e^2 \in \overline{R}$ such that $e \in (x - x^2)\overline{R}$. So, $e = (x - x^2)y$ for some $y \in \overline{R}$. Since $e$ is central, we have $e = \bar{e}x \cdot e(1-x) \cdot ey$, so $ex, e(1-x) \in U(e\overline{R}e)$. So we have, $ex + e(1-x) = e$, which contradicts Lemma \ref{exe}(1). Therefore, $\overline{R}$ is a Boolean ring.
\end{proof}

\begin{example}
Suppose $R$ is an Artinian ring. Then, the following conditions are equivalent:

(1) $R$ is a UQ ring.

(2) $R$ is a UJ ring.

(3) $R$ is a UU ring.
\end{example}

\begin{proof}
Using \cite[Corollary 6]{camilo}, we know that every Artinian ring is clean ring. Therefore, by Corollary \ref{exe 2}, we have that $R$ is a UQ ring uniquely when $R$ is a UJ ring. Also, since $R$ is Artinian, we have $J(R) \subseteq N(R)$. Hence, referring to \cite[Theorem 2.4(2)]{DL}, $R$ is a UU ring exactly when $R/J(R)$ is a UU ring. Consequently, taking into account Corollary \ref{exe 2}, $R$ is a UJ ring precisely when $R$ is a UU ring, as desired.
\end{proof}

\begin{corollary}
Let $R$ be finite ring. Then, the following conditions are equivalent:

(1) $R$ is a UQ ring.

(2) $R$ is a UJ ring.

(3) $R$ is a UU ring.
\end{corollary}

\begin{corollary}
For a ring $R$, the following conditions are equivalent:

(1) $R$ is a potent UQ ring.

(2) $R$ is a $J$-clean ring.
\end{corollary}

\begin{proof}
It is straightforward that (2) assures (1).

To show the opposite that (1) implies (2), suppose $R$ is a potent UQ ring. Applying Corollary \ref{exe 2}, we know that $R/J(R)$ is Boolean. Therefore, for each $a \in R$, we have $a - a^2 \in J(R)$. Since $R$ is a potent ring, there exists an idempotent $e \in R$ such that $a - e \in J(R)$. Thus, $R$ is a J-clean ring, as wanted.
\end{proof}

\section{Group Rings}

We begin here with the following preliminaries.

\begin{lemma} \label{QNR}
Let $R$ be a ring, $I$ an ideal of $R$ such that $I \subseteq J(R)$, and $\overline{R} = R/I$. Then, the following are valid:

(1) For every $\bar{q} \in QN(\overline{R})$, we have $q \in QN(R)$.

(2) For every $\bar{q} \in QN(\overline{R})$ and $p \in I$, we have $q + p \in QN(R)$.
\end{lemma}

\begin{proof}
(1) Let us assume $\bar{q} \in QN(\overline{R})$. If $ pq = qp$, then one follows that $\bar{q}\bar{p} = \bar{p}\bar{q}$. Therefore, $\bar{1} - \bar{p}\bar{q} \in U(\overline{R})$ implying that $1 - pq \in U(R)$. Hence, $q \in QN(R)$.

(2) For each $p \in I$, we can write $q + I = (q + p) + I \in QN(\overline{R})$. Thus, using property (1), we can conclude that $p + q \in QN(R)$.
\end{proof}

\begin{corollary}\label{factor UQ}
Let $R$ be a ring, $I \subseteq J(R)$, and $\overline{R} = R/I$ be a UQ ring. Then, $R$ is a UQ ring.
\end{corollary}

\begin{proof}
For any $u \in U(R)$, we have $\bar{u} \in U(\overline{R})$. Since $\overline{R}$ is a UQ ring, we can write $\bar{u} = \bar{1} + \bar{q}$, where $\bar{q} \in QN(\overline{R})$. Now, Lemma \ref{QNR}(1) is applicable to get that $q \in QN(R)$. Consequently, $u - (1 + q) \in I$. Thus, there exists $p \in I$ such that $u = 1 + (p + q)$. But, consulting with Lemma \ref{QNR}(2), we deduce $p + q \in QN(R)$. Therefore, $R$ is a UQ-ring, as stated.
\end{proof}

\begin{lemma}\label{sequ}
Let $R$ be a UQ-ring, $a \in U(R)$, and $f_n = 1 + a + a^2 + \cdots + a^n$. Then, the following two points occur:

(1) If $n$ is even, then $f_n \in U(R)$.

(2) If $n$ is odd, then $f_n \in QN(R)$.
\end{lemma}

\begin{proof}
(1) Suppose $n$ is even. We can rewrite $f_n$ as $$f_n = 1+ (1 + a)(a + a^3 + \cdots + a^{n-1}).$$ Since $a \in U(R)$, we know $a+1 \in QN(R)$, so that \cite[Proposition 2.10]{Cui} can be applied to infer that $$(1 + a)(a + a^3 + \cdots + a^{n-1}) \in QN(R).$$ Hence, $f_n \in U(R)$.

(2) Suppose $n$ is odd. We can rewrite $f_n$ as $f_n = 1+ af_{n-1}$. So, the usage of (1) guarantees that $f_{n-1} \in U(R)$, and therefore $$f_n = 1 + a f_{n-1} \in 1 + U(R) = QN(R),$$ as asked for.
\end{proof}

We now have all the ingredients necessary for proving the following major statement.

\begin{theorem}\label{UQgroupring}
Let $R$ be a ring and let $G$ be a group. If $RG$ is a UQ-ring, then $R$ is a UQ-ring and $G$ is a $2$-group.
\end{theorem}

\begin{proof}
Since $R$ is a good subring of $RG$, it follows directly from Lemma \ref{subring} that $R$ is a UQ-ring.

Let us now assume $g \in G$. Then, $R\left\langle g \right\rangle$ is also a good subring of $RG$. Thus, $R\left\langle g \right\rangle$ is a UQ-ring. If, however, we assume in a way of contradiction that $\left\langle g \right\rangle$ is an infinite cyclic group, then the application of Lemma \ref{sequ}(1) insures that $1 + g + g^2 \in U(RG)$. Hence, there exist integers, say $a \le b$, and $c_i \in R$, with the property
$c_a \neq 0 \neq c_b$ such that $$1 = \sum_{i=a}^{b} c_ig^i (1 + g + g^2) = c_ag^a + \sum_{a+1}^{b+1} d_ig^i + c_bg^{b+2}$$ for suitable $d_i \in R$, $i \in  \mathbb{Z}$. It is not too hard to verify that this gives the expected contradiction.

So, if there exists an element of $G$ such that its order is divisible by an odd prime, say $p$, then there will exist an element $g$ of $G$ with ${\rm order}(g) = p$. But since $\sum_{i=0}^{p-1}g^i \in U(RG)$ in view of Lemma \ref{sequ}(1), and obviously $(1-g)\sum_{i=0}^{p-1}g^i=0$, we detect that $1 -g = 0$, i.e., $g=1$, again a contradiction, as needed.
\end{proof}

\begin{lemma} \label{Delta subset Jaco}
Let $G$ be a locally finite $2$-group and let $R$ be a UQ ring. Then, $\Delta(RG) \subseteq J(RG)$.
\end{lemma}

\begin{proof}
Suppose that $\overline{R}$ equals to $R/J(R)$. Then, employing Lemma \ref{2 in J(R)}, $\bar{2} \in N(\overline{R})$. Therefore, with \cite[Proposition 16]{con} in hand, we derive $\Delta(\overline{R}G) \subseteq J(\overline{R}G)$.

On the other hand, with \cite[Lemma 4]{chennot} in mind, we deduce $J(R)G \subseteq J(RG)$, so one checks that $$J(\overline{R}G)=J(R/J(R)G) \cong J(RG/J(R)G)=J(RG)/J(R)G.$$
However, for each $f=\sum a_g(1-g) \in \Delta(RG)$, it is readily seen that $$\sum \Bar{a_g}(1-g) \in  \Delta(\overline{R}G) \subseteq J(RG)/J(R)G.$$ Hence, there exists $j \in J(RG)$ such that $f-j \in J(R)G \subseteq J(RG)$, and because $J(RG)$ is an ideal, we freely may conclude that $f$ belongs to $J(RG)$, completing the arguments after all.
\end{proof}

We are now in a position to prove validity of the following sufficient condition.

\begin{proposition}\label{locallyfinite}
If $R$ is a UQ ring and $G$ is a locally finite $2$-group, then $RG$ is a UQ ring.
\end{proposition}

\begin{proof}
With Lemma \ref{Delta subset Jaco} into account, we have $\Delta(RG) \subseteq J(RG)$ and, on the other side, we have $RG/\Delta(RG) \cong R$. Therefore, $RG/\Delta(RG)$ is a UQ ring. Consequently, Lemma \ref{factor UQ} forces that the ring $RG$ is a UQ ring.
\end{proof}

It has been proven in \cite{kosangroup} that if $R$ is a UJ ring and $G$ a locally finite $2$-group, then $RG$ is a UJ ring. We now attempting to reprove the same result approaching a different method which is definitely more concise and transparent.

\begin{theorem}\label{UJgroupring}
If $R$ is a UJ ring and $G$ is a locally finite $2$-group, then $RG$ is a UJ ring.
\end{theorem}

\begin{proof}
Suppose $R$ is a UJ ring and $G$ is a locally finite $2$-group. Now, Lemma \ref{2 in J(R)} allows us to have $2 \in J(R)$. Thus, assuming $\overline{R} = R/J(R)$, it must be that $\bar{2} \in N(\overline{R})$. Therefore, in virtue of \cite[Proposition 16]{con}, $\Delta(\overline{R}G)$ is a nil-ideal guaranteeing the inclusion $\Delta(\overline{R}G) \subseteq J(\overline{R}G)$. But since $\overline{R}G/\Delta(\overline{R}G) \cong \overline{R}$, the utilization of \cite[Proposition 1.3(5)]{kosan1} enables us that $\overline{R}G$ is a UJ ring.

Furthermore, because of the isomorphism $\overline{R}G \cong RG/J(R)G$, and with the aid of \cite[Lemma 4]{chennot}, we obtain the inclusion $J(R)G \subseteq J(RG)$. Hence, thanks to Proposition \cite[Proposition 1.3(5)]{kosan1}, $RG$ is a UJ ring, as pursued.
\end{proof}

\section{Some Specific Classes of UQ Rings}

In this section, we first introduce the rings $A_{n,m}(R)$, $B_{n,m}(R)$, and $C_{n}(R)$ for each $m, n \in \mathbb{N}$, and using these rings, we present certain new classes of UQ rings.

In fact, letting $R$ be an arbitrary ring, we define the aforementioned rings as follows:
\begin{align*}
&A_{n,m}(R) =R[x,y | x^n=xy=y^m=0], \\
&B_{n,m}(R) =R\left\langle x,y | x^n=xy=y^m=0  \right\rangle, \\
&C_{n}(R) =R \langle x,y | x^2=\underbrace{xyxyx...}_{\text{$n-1$ words}}=y^2=0 \rangle.
\end{align*}

In another vein, Wang introduced in \cite{wang} the following matrix ring $S_{n,m}(R)$. Supposing $R$ is a ring, the matrix ring $S_{n,m}(R)$ can be represented as

$$\left\{ \begin{pmatrix}
   a & b_1 & \cdots & b_{n-1} & c_{1n} & \cdots & c_{1 n+m-1}\\
   \vdots  & \ddots & \ddots & \vdots & \vdots & \ddots & \vdots \\
   0 & \cdots & a & b_1 & c_{n-1,n} & \cdots & c_{n-1,n+m-1} \\
   0 & \cdots & 0 & a & d_1 & \cdots & d_{m-1} \\
   \vdots  & \ddots & \ddots & \vdots & \vdots & \ddots & \vdots \\
   0 & \cdots & 0 & 0  & \cdots & a & d_1 \\
   0 & \cdots & 0 & 0  & \cdots & 0 & a
\end{pmatrix}\in T_{n+m-1}(R) : a, b_i, d_j,c_{i,j} \in R \right\}$$

Also, let

$$T_{n,m}(R)=\left\{ \left(\begin{array}{@{}c|c@{}}
  \begin{matrix}
  a & b_1 & b_2 & \cdots & b_{n-1} \\
  0 & a & b_1 & \cdots & b_{n-2} \\
  0 & 0 & a & \cdots & b_{n-3} \\
  \vdots & \vdots & \vdots & \ddots & \vdots \\
  0 & 0 & 0 & \cdots & a
  \end{matrix}
  & \bigzero \\
\hline
  \bigzero &
  \begin{matrix}
  a & c_1 & c_2 & \cdots & c_{m-1} \\
  0 & a & c_1 & \cdots & c_{m-2} \\
  0 & 0 & a & \cdots & c_{m-3} \\
  \vdots & \vdots & \vdots & \ddots & \vdots \\
  0 & 0 & 0 & \cdots & a
  \end{matrix}
\end{array}\right)\in T_{n+m}(R) : a, b_i,c_j \in R \right\} $$

and

$$U_{n}(R)=\left\{ \begin{pmatrix}
   a & b_1 & b_2 & b_3 & b_4 & \cdots & b_{n-1} \\
   0 & a & c_1 & c_2 & c_3 & \cdots & c_{n-2} \\
   0 & 0 & a & b_1 & b_2 & \cdots & b_{n-3} \\
   0 & 0 & 0 & a & c_1 & \cdots & c_{n-4} \\
   \vdots & \vdots & \vdots & \vdots &  &  & \vdots \\
   0 &0 & 0 & 0 & 0 & \cdots & a
\end{pmatrix}\in T_{n}(R) :  a, b_i, c_j \in R \right\}$$

We are now ready to demonstrate in the following assertion the relationships between these rings. Specifically, we are able to establish the following technical claim.

\begin{lemma} \label{section 3 lemma}
Let $R$ be a ring and $m, n \in \mathbb{N}$. Then, the next three isomorphisms of rings are fulfilled:

(1) $A_{n,m}(R) \cong T_{n,m}(R)$.

(2) $B_{n,m}(R) \cong S_{n,m}(R)$.

(3) $C_{n}(R) \cong U_{n}(R)$.
\end{lemma}

\begin{proof}
(1) We write $$f = a+ \sum_{i=1}^{n-1}b_ix^i + \sum_{j=1}^{m-1}c_jy^j \in A_{n,m}(R)$$ and define $\varphi:  A_{n,m}(R) \to  T_{n,m}(R)$ as
$$\varphi(f)=
\left(\begin{array}{@{}c|c@{}}
  \begin{matrix}
  a & b_1 & b_2 & \cdots & b_{n-1} \\
  0 & a & b_1 & \cdots & b_{n-2} \\
  0 & 0 & a & \cdots & b_{n-3} \\
  \vdots & \vdots & \vdots & \ddots & \vdots \\
  0 & 0 & 0 & \cdots & a
  \end{matrix}
  & \bigzero \\
\hline
  \bigzero &
  \begin{matrix}
  a & c_1 & c_2 & \cdots & c_{m-1} \\
  0 & a & c_1 & \cdots & c_{m-2} \\
  0 & 0 & a & \cdots & c_{m-3} \\
  \vdots & \vdots & \vdots & \ddots & \vdots \\
  0 & 0 & 0 & \cdots & a
  \end{matrix}
\end{array}\right).$$
It plainly can be shown via some routine technical tricks that $\varphi$ is a ring isomorphism, which fact we leave for a direct inspection by the interested reader.

(2) We write $f \in B_{n,m}(R)$ such that
\begin{align*}
    f &=  a_{00}y^0x^0 + a_{01}y^0x^1 + \cdots + a_{0,n-1}y^0x^{n-1}  \\
      &+  a_{10}y^1x^0 + a_{11}y^1x^1 + \cdots + a_{1,n-1}y^1x^{n-1}   \\
      & \qquad \vdots   \qquad \qquad \vdots  \qquad \qquad \qquad \qquad  \vdots   \\
      &+ a_{m-1,0}y^{m-1}x^0 + a_{m-1,1}y^{m-1}x^1 + \cdots + a_{m-1,n-1}y^{m-1}x^{n-1}    \\
\end{align*}
and define $\psi :  B_{n,m}(R) \to  S_{n,m}(R)$ as
$$\psi(f)=
 \begin{pmatrix}
   a_{00} & a_{10} & \cdots & a_{m-1,0} & a_{m-1,1} & \cdots & a_{m-1,n-1}\\
   \vdots  & \ddots & \ddots & \vdots & \vdots & \ddots & \vdots \\
   0 & \cdots & a_{00} & a_{10} & a_{11} & \cdots & a_{1,n-1} \\
   0 & \cdots & 0 & a_{00} & a_{01} & \cdots & a_{0,n-1} \\
   \vdots  & \ddots & \ddots & \vdots & \vdots & \ddots & \vdots \\
   0 & \cdots & 0 & 0  & \cdots & a_{00} & a_{0,1} \\
   0 & \cdots & 0 & 0  & \cdots & 0 & a_{00}
\end{pmatrix}.$$

It routinely follows that the so defined map is a ring isomorphism, which verification we omit to be done by the interested reader.

(3) We write the coefficients as follows
$$f= \sum_{0\le i_j \le 1\atop 1 \le j \le n-1}d_{(i_1, \dots , i_{n-1})}\underbrace{y^{i_1}x^{i_2}y^{i_3}x^{i_4}...}_{\text{$n-1$ words}}\in C_{n}(R)$$

and define $\phi :  C_{n}(R) \to  S_{n,m}(R)$ as
$$\phi(f)=
\begin{pmatrix}
   d_{(0,0,0, \dots,0)} & d_{(1,0,0, \dots,0)} & d_{(1,1,0, \dots,0)} & d_{(1,1,1, \dots,0)} & \cdots & d_{(1,1,1, \dots,1)}\\
   0 & d_{(0,0,0, \dots,0)} & d_{(0,1,0, \dots,0)} & d_{(0,1,1, \dots,0)} & \cdots & d_{(0,1,1, \dots,1)} \\
   0 & 0 & d_{(0,0,0, \dots,0)} & d_{(1,0,0, \dots,0)} & \cdots & d_{(1,\dots,1,0,0)}\\
   0 & 0 & 0 & d_{(0,0,0, \dots,0)} & \cdots & d_{(0,1,\dots,1,0,0)} \\
   \vdots  & \vdots & \vdots & \vdots & \ddots & \vdots \\
   0 &  0 & 0  & 0  & \cdots & d_{(0,0,0, \dots,0)}
\end{pmatrix}.$$
A simple game with matrices is only needed in order to check that the so defined map is an isomorphism of rings, which technical manipulation we leave to the interested reader.
\end{proof}

\begin{example}
Let $R$ be a ring, then the following three items hold:

(1) $R\left[ x,y | x^2=xy=y^2=0  \right] \cong \left\{
\begin{pmatrix}
    a_1 & a_2 & 0 & 0 \\
    0 & a_1 & 0 & 0 \\
    0 & 0 & a_1 & a_3 \\
    0 & 0 & 0 & a_1
\end{pmatrix} : a_i \in R \right\}
$

(2) $R\left\langle x,y | x^2=xy=y^2=0  \right\rangle \cong \left\{
\begin{pmatrix}
    a_1 & a_2 & a_3 \\
    0 & a_1 & a_4 \\
    0 & 0 & a_1
\end{pmatrix} : a_i \in R \right\} $

(3) $R\left\langle x,y | x^2=xyx=y^2=0  \right\rangle \cong \left\{
\begin{pmatrix}
    a_1 & a_2 & a_3 & a_4 \\
    0 & a_1 & a_5 & a_6 \\
    0 & 0 & a_1 & a_2 \\
    0 & 0 & 0 & a_1
\end{pmatrix} : a_i \in R \right\} \cong T(T(R,R), M_2(R))$
\end{example}

We now define, respectively, the following corresponding sets for the subsets $I$, $J$ and $K$ of the rings $A_{n,m}(R)$, $B_{n,m}(R)$ and $C_{n}(R)$ thus:
\begin{align*}
    &\nabla_I(A_{n,m}(R))= \left\{ a+ \sum_{i=1}^{n-1}b_ix^i + \sum_{j=1}^{m-1}c_jy^j \in A_{n,m}(R) \mid a \in I \right\}, \\
    &\nabla_J(B_{n,m}(R))= \left\{ \sum_{i=0}^{m-1} \sum_{j=0}^{n-1} a_{ij}y^ix^j \in B_{n,m}(R) \mid a_{00} \in J \right\}, \\
    &\nabla_K(C_{n}(R)) = \left\{ a + C_{n}(R)x +C_{n}y \in C_{n}(R) \mid a \in K \right\}.
\end{align*}

Next, Lemma \ref{section 3 lemma} helps us to illustrate the truthfulness of the following consequence.

\begin{corollary} \label{unit}
Let $R$ be a ring. Then, the next three statements are valid:

(1) $U(A_{n,m}(R))=\nabla_{U(R)}(A_{n,m}(R)),$

(2) $U(B_{n,m}(R))=\nabla_{U(R)}(B_{n,m}(R)),$

(3) $U(C_{n}(R))=\nabla_{U(R)}(C_{n}(R)).$

\end{corollary}

We close our work with the following assertion whose proof is removed being almost obvious bearing in mind the quoted above claims.

\begin{lemma}
Let $R$ be a ring. Then, the following three points are true:

(1) $A_{n,m}(R)$ is a UQ if, and only if, $R$ is a UQ ring.

(2) $B_{n,m}(R)$ is a UQ if, and only if, $R$ is a UQ ring.

(3) $C_{n}(R)$ is a UQ  if, and only if, $R$ is a UQ ring.
\end{lemma}

\end{document}